\documentclass[12pt,oneside]{amsart}
\usepackage[a4paper, total={4.5in, 7.5in}]{geometry}

\usepackage{amsthm,amsmath,amssymb}
\usepackage{mathrsfs}
\usepackage{enumerate}
\usepackage{amsfonts}
\usepackage{verbatim}
\usepackage{amsbsy}
\usepackage{amsmath}
\usepackage{amssymb}
\usepackage{MnSymbol}
\usepackage{mathrsfs}
\usepackage{xcolor}
\usepackage{cite}
\usepackage[super]{nth}

\newtheorem{theorem}{Theorem}[section]
\newtheorem{lemma}[theorem]{Lemma}
\newtheorem{proposition}[theorem]{Proposition}

\newtheorem*{claim}{Claim}

\newtheorem{corollary}[theorem]{Corollary}
\theoremstyle{definition}
\newtheorem{definition}[theorem]{Definition}

\theoremstyle{remark}
\newtheorem{remark}[theorem]{Remark}
\newtheorem{example}[theorem]{Example}

\DeclareMathOperator{\FR}{FR}

\DeclareMathOperator{\id}{id}
\DeclareMathOperator{\fr}{FR}
\DeclareMathOperator{\ot}{OT}

\DeclareMathOperator{\Th}{Th}
\DeclareMathOperator{\subalg}{Sub}
\DeclareMathOperator{\ult}{Ult}

\begin{document}

\title[Structural Considerations of Ramsey Algebras]{Structural Considerations of Ramsey Algebras}
\author{Zu Yao Teoh}
\address{School of Mathematical Sciences\\
Universiti Sains Malaysia\\
11800 USM, Malaysia}
\email{teohzuyao@gmail.com}
\date{\today}

\begin{abstract}
Ramsey algebras is an attempt to investigate Ramsey spaces generated by algebras in a purely combinatorial fashion. Previous studies have focused on the basic properties of Ramsey algebras and the study of a few specific examples. In this article, we study the properties of Ramsey algebras from a structural point of view. For instance, we will see that isomorphic algebras have the same Ramsey algebraic properties, but elementarily equivalent algebras need not be so, as expected. Answer to a long-standing question about the Cartesian products of Ramsey algebras is also given.
\end{abstract}

\maketitle


\section{Introduction}
In \cite{carlson1988some}, Carlson introduced the notion of a Ramsey space, now called topological Ramsey space following Todorocevic extension of the work \cite{todorocevic}. When the Ramsey space is generated by an algebra, Carlson has suggested that  a purely combinatorial approach might be possible. In his doctoral work, Teh pursued this theme and the study has since been known as Ramsey algebra.

An algebra is a structure consisting of a family of sets, called the domains of the algebra, and a family of operations on these sets. For instance, groups $(G, \circ)$ as we know them are algebras. So are rings and fields $(R, +, \times)$. Groups, rings, and fields are what we refer to as \emph{homogeneous} algebras. Such algebras have only one type of domain, $G$ or $R$ as indicated above. Such is contrasted with \emph{heterogeneous} algebras, a typical example which is a vector space, where its domains consist of the set of vectors as well as the set of scalars. Of course, homogeneous algebras are special cases of heterogeneous algebras.

Our concern in this paper is solely on homogeneous algebras. Arguably, Ramseyan theme combinatorics on algebras has roots in Hindman's Theorem. It states that, for each positive integer $r$ and each coloring $c:\mathbb{Z}^+\to\{1, \ldots, r\}$, there exists an infinite $S\subseteq\mathbb{Z}^+$ such
that $c$ is constant on the set $\text{FS}=\left\{\sum_{i\in F} i:F\subseteq S, F\:\text{is finite}\right\}$. The same theorem can be cast in terms of finite
subsets of the natural numbers \cite{milliken1975ramsey}. Considering questions about infinite sets along the same line, a result of Erd\"{o}s-Rado
\cite{erdos1952combinatorial} shows that not all sets of reals possess similar property. The counterexample considered by Erd\"{o}s-Rado made use of the axiom of choice and is hence non-constructive. As a result, the question arises as to whether sets definable in some ways can possess some Ramseyan property. Galvin-Prikry
\cite{GP73} showed that the Borel sets do, whereas, using the method of forcing, Silver \cite{silver1970every} showed that the analytic sets have the
desired property.

Ellentuck \cite{ellentuck1974new} came into the scene by showing that the result of Silver could also be proved using a topological formulation without having to appeal to
metamathematical methods. In his proof, Ellentuck introduced what is to be called the Ellentuck topology. Seeing the potential in such an argument, Carlson \cite{carlson1988some} arrived at a generalization, in which he introduced and developed the notion of a (topological) Ramsey space. The power of Ramsey spaces lies in its ability to derive as corollaries classical
results such as the Hales-Jewett Theorem and Hindman's Theorem, which Carlson provided in that same paper.

If one is to ask whether a given structure is a Ramsey space, the definition requires that one checks some topological properties of the space. However,
Carlson's abstract version of Ellentuck's Theorem turns a topological question into a combinatorial one (cf. \cite{carlson1988some} and \cite{teoh2016vector}). Some early works on Ramsey algberas include \cite{teh2016ramsey}, \cite{wcT13b}, \cite{wcTIUF}, and \cite{teohoctonions}. Two articles on heterogeneous Ramsey algebras can be found in \cite{teoh2016vector} and \cite{teoh2017matrices}. The two doctoral theses \cite{tehthesis} and \cite{teohthesis} may also be of interest.

In this paper, we investigate some structural questions pertaining to Ramsey algebras. In particular, we study the property of being Ramsey under the notions of elementary equivalence and elementary extension. The direct limit of a family of Ramsey algebras is also investigated in conjunction with that. We also give a partial answer to a question since the inception of Ramsey algebras, a question that pertains to the Cartesian products of Ramsey algebras.

We fix some symbols and conventions. $\omega$ will denote the set of natural numbers $0, 1, 2, \ldots$. If $D$ is a set, then $\id_D$ denotes the identity function on $D$. If $\sigma$ is a finite sequence, then $|\sigma|$ denotes the length of the sequence. If $A$ is
a nonempty set, an infinite sequence of $A$ is an element of ${^\omega}A$. A sequence, finite or infinite, will be denoted by an arrow over a Roman alphabet such as $\vec{b}$; if the elements are to be listed, they will be enclosed in angled brackets such as $\langle a_0, a_1, \ldots\rangle$. An $n$-tuple will be denoted with an overbar or, when brevity permits, a lower case Greek alphabet. When we have a finite sequence $\vec{b}=\langle a_1, a_2, \ldots, a_n\rangle$, then $\bar{b}$ will be understood to be $(a_1, a_2, \ldots, a_n)$; in our context, this happens predominantly when we discuss orderly composition and reduction (see Definitions \ref{ordcomp} and \ref{reduction}).

We assume the axiom of choice throughout.

\section{Ramsey Algebras}
As mentioned in the previous section, we will only be concerned with homogeneous algebras in this paper, hence we will circumvent the more general notions pertaining to heterogeneous algebras. The interested reader can find the more general treatment in Carlson's original paper \cite{carlson1988some} as well as in \cite{teoh2016vector} and \cite{teoh2017matrices}.

We assume that the reader is familiar with first-order logic. In particular, we will not give the precise definitions of elementary equivalence, elementary extension, and direct limit. Now, from a logical point of view, an algebra is a structure interpreting a purely functional
language. That is, an algebra $\mathcal{A}$ consists of a universe $A$ and a family $\mathcal{F}$ of finitary operations (i.e. functions)\footnote{We will use the terms function and operations interchangeably throughout this paper, as is customary with the practice of algebraists.} on $A$, i.e. each $f\in\mathcal{F}$ has as domain some finite Cartesian product $A^n$ and codomain $A$. We write $\mathcal{A}=(A, \mathcal{F})$ for the algebra just
described. An algebra is said to be \emph{infinite} if its universe is infinite; it is finite otherwise. An algebra whose family of operations consist only of unary functions will be called a \emph{unary system} or \emph{unary algebra}.

Our first formal definition deals with a certain type of composition of operations. The definition here is a slight modification of the original one given in Definition 3.6 by Carlson \cite{carlson1988some}.

\begin{definition}[Orderly Composition \& Orderly Terms]\label{ordcomp}
Let $\mathcal{F}$ be a family of operations on $A$. An $n$-ary function $f$ is called an \emph{orderly composition} of $\mathcal{F}$ if
there exists $h_1, h_2, \ldots, h_k, g\in\mathcal{F}$ such that
\begin{enumerate}
\item $g$ is a $k$-ary function,
\item $h_j$ is an $n_j$-ary function for each $j\in\{1, 2, \ldots, k\}$,
\item $\sum_{j=1}^k n_j=n$, and
\item if $\bar{x}_1=(x_1, x_2, \ldots, x_{n_1})$ and for each $j=\{2$,
    $\ldots$, $k\}$ we have $\bar{x}_j=\left(x_{\left(\sum_{i=1}^{j-1}n_i\right)+1}, x_{\left(\sum_{i=1}^{j-1}n_i\right)+2}, \ldots, x_{\sum_{i=1}^jn_i}\right)$, then $f(x_1, \ldots, x_n)=g(h_1(\bar{x}_1), h_2(\bar{x}_2), \ldots, h_k(\bar{x}_k))$.
\end{enumerate}

The collection $\ot(\mathcal{F})$ of \emph{orderly terms} over $\mathcal{F}$ is the \emph{smallest} collection of operations containing
$\mathcal{F}\cup\left\{\id_A\right\}$ and is closed under orderly composition.
\end{definition}

The collection of orderly terms over $\mathcal{F}$ is in fact the collection of operations on $A$ which can be generated by an
application of finitely many of the following more lucid rules:
\begin{enumerate}
\item $\text{id}_A$ is an orderly term,
\item every operation  in $\mathcal{F}$ is an orderly term,
\item if $f$ is an operation on $A$ given by $f(\bar{x}_1, \bar{x}_2, \ldots, \bar{x}_k)=g(h_1(\bar{x}_1)$, $h_2(\bar{x}_2)$, $\ldots$, $h_k(\bar{x}_k))$ for some
    $g\in\mathcal{F}$ and some orderly terms $h_1$, $h_2$, $\ldots$, $h_k$, then $f$ is an orderly term.
\end{enumerate}

In formal terms, suppose a function symbol is associated with each element of $\mathcal{F}$. The orderly terms over $\mathcal{F}$ are defined by formal terms in which each variable in the term occurs exactly once and the variables appear ``in order."

We denote the concatenation operation by $\ast$ in this paper.

\begin{definition}[Reduction $\leq_\mathcal{F}$]\label{reduction}
Let $(A, \mathcal{F})$ be an algebra and let $\vec{a}$ and $\vec{b}$ be members of ${^\omega}A$. Then $\vec{a}$ is said to be a \emph{reduction} of $\vec{b}$ if there exist orderly terms $f_j$ over $\mathcal{F}$ and finite subsequences $\vec{b}_j$
of $\vec{b}$ such that
\begin{enumerate}
\item $\vec{a}(j)=f_j(\bar{b}_j)$ for each $j\in\omega$ and
\item $\vec{b}_0\ast\vec{b}_1\ast\cdots$ forms a subsequence of $\vec{b}$.
\end{enumerate}
We write $\vec{a}\leq_\mathcal{F}\vec{b}$ to mean $\vec{a}$ is a reduction of $\vec{b}$.
\end{definition}

The relation $\leq_\mathcal{F}$ is a preorder\footnote{A preorder is a relation that is reflexive and transitive.} on ${^\omega}A$. Note that the inclusion of the identity functions in the set of orderly terms is necessary to ensure that the relation
$\leq_\mathcal{F}$ is reflexive.

We pause to illustrate the two definitions above using the algebra $(\mathbb{Z}^+, +)$. (According to our notation above, the exact notation for this algebra should be written as $(\mathbb{Z}^+, \{+\})$. However, from this point onwards, we will drop the curly brackets encompassing the operations when the list is short.) Examples of orderly terms over $\{+\}$ include $+(x_0, +(x_1, x_2))$ and $+(+(x_1, x_2), +(x_3, +(x_4, x_5)))$. Note that the variables appear in order from left to right and no repetition occurs. Also from this point onward, we will write orderly terms involving common operations such as $+$ is the more suggestive manner such as $x_0+(x_1+x_2)$ and $(x_1+x_2)+(x_3+(x_4+x_5))$ for the orderly terms above. Of course, since addition on $\mathbb{Z}^+$ is commutative, the bracketing does not matter, and we will even be omitting them and write $x_0+x_1+x_2$ and $x_1+x_2+x_3+x_4+x_5$. We will revert to conventional notations whenever possible.

As for reduction, let $\vec{b}=\langle 1, 3, 5, 7, \ldots\rangle$. Then an example of a reduction of $\vec{b}$ is $\langle 6, 16, 53, 23, \ldots\rangle$ with the associated finite sequences and orderly terms given respectively by $x_1+x_2$ on $\langle 1, 5\rangle$, $x_1+x_2$ again but on $\langle 7, 9\rangle$, $x_1+x_2+x_3$ on $\langle 15, 17, 21\rangle$, and $\id_{\mathbb{Z}^+}(x)$ on $23$, and so on. We have reverted to conventional notation in this example.

For families consisting of unary operations, the orderly terms coincide with the familiar composite functions:
\begin{remark}
If $\mathcal{F}$ is a family of unary operations on $A$, then $\ot(\mathcal{F})$ coincides with the set of all composites of functions in $\mathcal{F}\cup\{\id_A\}$.
\end{remark}

We continue with more definitions for Ramsey algebras. The next one is the generalization of the sets analogous to $\text{FS}$ in the introductory section and are precisely those sets $\text{FS}(\langle x_n\rangle_{n=1}^\infty)$ appearing in Chapter 5 of \cite{hindman1998algebra} in the context of the algebra $(\mathbb{Z}^+, +)$.

\begin{definition}\label{FR}
If $\vec{b}$ is an an infinite sequence of $A$, then
\begin{eqnarray*}
\FR_\mathcal{F}(\vec{b}) &=& \{\vec{a}(0):\vec{a}\leq_\mathcal{F}\vec{b}\} \\
&=& \{f(\sigma):f\in\ot(\mathcal{F}), \sigma\;\text{is a finite sequence of}\:\vec{b}\}.
\end{eqnarray*}
\end{definition}

We end this section with the central notion of the paper.

\begin{definition}[Ramsey Algebra]\label{RA}
An algebra $(A, \mathcal{F})$ is said to be a Ramsey algebra if, for each infinite sequence
$\vec{b}$ and each $X\subseteq A$, there exists $\vec{a}\leq_\mathcal{F}\vec{b}$ such that $\FR_\mathcal{F}(\vec{a})$ is
either contained in or disjoint from $X$.

Such a sequence $\vec{a}$ is said to be \emph{homogeneous} for $X$ (with respect to $\mathcal{F}$).
\end{definition}

In the language of Ramsey algebras, Hindman's theorem takes the following form (see Corollary 5.9 of \cite{hindman1998algebra}).

\begin{theorem}[Hindman]\label{semigroup}
Every semigroup is a Ramsey algebra.
\end{theorem}

As mentioned earlier, the origin of Ramsey algebras has its roots in the notion a Ramsey space introduced by Carlson. Readers who are interested in the exact connection between (topological) Ramsey spaces and Ramsey algebras are referred to \cite{teoh2016vector}.

Carlson's theorem on the variable words, from which he derived many of the classical combinatorial theorems mentioned in the introduction, can be stated in terms of Ramsey algebras:

\begin{theorem}[Carlson]
The algebra of variable words with finite alphabets equipped with the operations of ``substitution'' and ``evaluation'' is a Ramsey algebra.
\end{theorem}

Before we embark on an investigation of new questions, we state a few more results from past works.

\begin{theorem}
No infinite division ring is a Ramsey algebra. In
particular, no infinite integral domain is a Ramsey algebra.
\end{theorem}

\begin{corollary}\label{onrings}
No infinite ring with multiplicative identity having characteristic zero is a Ramsey algebra.
\end{corollary}

The theorem and its corollary above can be found in \cite{teh2016ramsey}. The next theorem, which is found in the same paper, will play an important role in Section 4.
\begin{theorem}[Characterization of Unary Ramsey Algebras]\label{unaryS}
Let $\mathcal{A}=(A, \mathcal{F})$ be a unary system. Then $\mathcal{A}$ is a Ramsey algebra if and only if for each $a\in A$, there exists an $F\in\ot(\mathcal{F})$ such that $F(a)\in\{a\in A:f(a)=a\;\text{for all}\;f\in\mathcal{F}\}$.
\end{theorem}

This theorem will be used in conjunction with the predecessor function in Section \ref{predfncsec} below. Points in the set $\{a\in A:f(a)=a\;\text{for all}\;f\in\mathcal{F}\}$ will be called \emph{fixed points} of $\mathcal{F}$ and $\ot(\mathcal{F})$ coincides with the smallest set of composite functions containing $\mathcal{F}$. Thus, when applying the theorem above, we will often speak of a point being able to be sent to the set of fixed points by \emph{finitely many applications} of the functions in $\mathcal{F}$.

\section{Homomorphic Algebras \& Quotients}
The main result of this section is that isomorphic algebras have the same Ramsey algebraic property. This is not surprising, but we state it formally as a consequence of epimorphic mappings.

Recall that the notion of isomorphic algebras only apply to algebras with the same language. For the convenience of writing proofs, we give the definition of the notion of a homomorphism. Hence, if $\mathcal{L}$ is a family of function symbols and $\mathcal{A}_0=(A_0, \mathcal{F}_0), \mathcal{A}_1=(A_1, \mathcal{F}_1)$ are algebras of the language $\mathcal{L}$, then $\pi:A_1\to A_0$ is said to be a \emph{homomorphism} from $A_1$ into $A_0$ if
\begin{equation}\label{homdef}
\pi(f^{\mathcal{A}_1}(b_1, \ldots, b_n))=f^{\mathcal{A}_0}(\pi(b_1), \ldots, \pi(b_n))
\end{equation}
for each $f\in\mathcal{L}$ and each $b_1, \ldots, b_n\in A_1$. It follows that, if $t$ is an $\mathcal{L}$-term and $b_1, \ldots, b_n\in A_1$, then
\begin{equation}\label{homeq}
\pi(t^{\mathcal{A}_1}(b_1, \ldots, b_n))=t^{\mathcal{A}_0}(\pi(b_1), \ldots, \pi(b_n)).
\end{equation}
In particular, Eq.\:\ref{homeq} holds for $\mathcal{L}$-terms interpreting orderly terms.

\begin{theorem}\label{epi}
If $\pi:A_1\to A_0$ is an epimorphism and $(A_1, \mathcal{F}_1)$ is a Ramsey algebra, then $(A_0, \mathcal{F}_0)$ is a Ramsey algebra.
\end{theorem}
\begin{proof}
Suppose $\pi:A_1\to A_0$ is an epimorphism and $(A_1, \mathcal{F}_1)$ a Ramsey algebra; let $\vec{b}$ be an infinite sequence $A_0$ and $X\subseteq A_0$. Set $\vec{\beta}$ to be such that, for each $i\in\omega$, $\vec{\beta}(i)$ is a representative of the preimage of $\vec{b}(i)$ under $\pi$, the exact representative which is immaterial, and let $Y=\{\alpha\in A_1:\pi(\alpha)\in X\}$. We may now choose an $\vec{\alpha}\leq_{\mathcal{F}_1}\vec{\beta}$ homogeneous for $Y$.

We claim that $\vec{a}=\langle\pi(\vec{\alpha}(i)):i\in\omega\rangle$ is the desired reduction of $\vec{b}$. To see this, first note that $\vec{a}$ is indeed a reduction of $\vec{b}$ by appealing to Eq.\:\ref{homeq}. Secondly, let $c\in\fr_{\mathcal{F}_0}(\vec{a})$; thus, let $t$ be a term of the language of the algebras interpreting an orderly term over $\mathcal{F}_0$ and let $n_1<\cdots<n_N$ be indices such that $c=t^{\mathcal{A}_0}(\pi(\vec{\alpha}(n_1)), \ldots, \pi(\vec{\alpha}(n_N)))$. By Eq.\:\ref{homeq}, we obtain $c=\pi(t^{\mathcal{A}_1}(\vec{\alpha}(n_1), \ldots, \vec{\alpha}(n_N)))$. Thus, we see that $c$ is the image of an element of $\fr_{\mathcal{F}_1}(\vec{\alpha})$ under $\pi$.

Therefore, $\fr_{\mathcal{F}_0}(\vec{a})\subseteq X$ or $\fr_{\mathcal{F}_0}(\vec{a})\subseteq A_0\setminus X$ depending respectively on whether $\fr_{\mathcal{F}_1}(\vec{\alpha})\subseteq Y$ or $\fr_{\mathcal{F}_1}(\vec{\alpha})\subseteq A_1\setminus Y$. Thus, the homogeneity of $\vec{a}$ for $X$ is established.
\end{proof}

\begin{corollary}\label{isothm}
The property of being a Ramsey algebra is an invariant under isomorphism.
\end{corollary}

The fact that isomorphic algebras have the same Ramsey algebraic property do not come as a surprise. In the next section, we will see that, on the other hand, elementary equivalence is not sufficient to warrant such invariance.

We end this section with quotient algebras. We learn from group theory and ring theory that being an equivalence relation on the elements of a group or ring itself does not warrant a well-defined quotient. In the case of groups, the condition on the relation is for the associated subgroup to be normal. In general, we need the equivalence relation to be a \emph{congruence}. A congruence relation $E$ on an algebra is one where every operation in the algebra is compatible with:

\begin{definition}[Compatible Relation \& Congruence]
Let $\mathcal{A}=(A, \mathcal{F})$ be an algebra, let $E$ be an equivalence relation on $|\mathcal{A}|$, and let $f\in\mathcal{F}$. We say that $E$ is compatible with $f$ if $f(a_1, a_2, \ldots, a_{|f|})Ef(b_1, b_2, \ldots, b_{|f|})$ whenever $a_iEb_i$ for each $i\in\{1, 2, \ldots, |f|\}$. We say that $E$ is a congruence relation on $\mathcal{A}$ if $E$ is compatible with every $f\in\mathcal{F}$.
\end{definition}

A congruence relation $E$ partitions the universe of an algebra $\mathcal{A}$ into classes in a natural way that results in the operations on $\mathcal{A}$ being preserved among the pieces in an inherent manner. To be precise, a congruence relation $E$ ensures that the quotient map $\rho:a\mapsto[a]$ from $\mathcal{A}$ onto its quotient by $E$ is a homomorphism (see the paragraph immediately following Definition 5.2 on page 36 of \cite{univalg}). Since such a homomorphism is epimorphic, we have the following:

\begin{theorem}
The quotient of a Ramsey algebra by a congruence relation is Ramsey.
\end{theorem}

An example of a congruence relation on an algebra is furnished by an ultrafilter, which we will have a chance to see in the next section.

\section{Cartesian Products, Ultrapowers, \& First-order Properties}\label{predfncsec}
The predecessor function $p$ on $\omega$ plays an important role in this section and it is defined by
\begin{displaymath}
   p(n) = \left\{
     \begin{array}{lr}
       0 & \text{if}\; n=0; \\
       n-1 & \text{otherwise}.
     \end{array}
   \right.
\end{displaymath}

$p$ has exactly one fixed point, namely $0$. Theorem \ref{unaryS} above in essence states that algebras whose families of operations are unary admit a simple characterization in terms of fixed points. Applying this characterization, we see that $\mathcal{A}=(\omega, p)$ is a Ramsey algebra. We will call this algebra the \emph{predecessor algebra} and we will denote it by $\mathcal{A}$ throughout this section. The predecessor algebra allows us to show that the following three conjectures, desirable as they may be, are not true:
\begin{enumerate}
\item Cartesian products of Ramsey algebras are Ramsey algebras.
\item Ultrapowers of Ramsey algebras are Ramsey.
\item Algebras that are elementarily equivalent have the same Ramsey algebraic property.
\end{enumerate}
We will show that these statements are false in this section. We want to make a formal emphasis before proceeding:
\begin{remark}
All algebras in this section are assumed to be \emph{infinite} unless otherwise stated, namely at the end of the section.
\end{remark}

The question on whether Cartesian products of Ramsey algebras are Ramsey is a question that has been asked since the inception of the subject. Surprisingly, this long-standing question can be easily answered by the predecessor algebra. We state the definition of a Cartesian product of algebras for convenience:

Given a family of algebras $\mathcal{A}_\xi=(A_\xi, \mathcal{F}_\xi)$ (with $\xi\in\kappa$, $\kappa$ some cardinal) of the same language $\mathcal{L}$, the Cartesian product $\mathcal{A}=(A, \mathcal{F})$ of this family is such that
\begin{enumerate}
\item[a.] the domain of $\mathcal{A}$ is the Cartesian product $\prod_{\xi<\kappa}A_\xi$ and
\item[b.] for each operation $f\in\mathcal{F}$, if $f$ interprets the function symbol $F\in\mathcal{L}$, then $f$ acts coordinate-wise and each coordinate, say the $\xi$th coordinate, is acted by the corresponding operation $f_\xi$, which also interprets $F$.
\end{enumerate}

\begin{proposition}
The infinite Cartesian product $\mathcal{B}=\prod_{i\in\omega}\mathcal{A}$ of $\mathcal{A}=(\omega, p)$ is not a Ramsey algebra.
\end{proposition}
\begin{proof}
We again exploit the characterization given by Theorem \ref{unaryS}. Consider the element $\omega=\langle0, 1, 2, \ldots\rangle\in\prod_{i\in\omega}\omega$ of the Cartesian product. (Note that we are using $\omega$ both as the set of natural numbers as well as a member of $\prod_{i\in\omega}\omega=\omega^\omega$.) It can not be sent to the fixed point $\varphi=\langle 0, 0, \ldots\rangle$ by any composition of the sole operation $f$ in the algebra $\mathcal{B}$ since each orderly term has the form $f^n$ and $f^n(\omega)=\langle 0, 0, \ldots, 0, 1, 2, 3, \ldots\rangle$, where the initial segment consisting of $0$'s has length $n+1$.
\end{proof}

Note that the infinity of the product is essential here. The case when the product is finite requires a slightly different algebra. Let $s$ be the successor function, let $F, G$ be a unary function symbols, and introduce the algebras $\mathcal{C}, \mathcal{D}$ with $F^{\mathcal{C}}=G^{\mathcal{D}}=d$ and $F^{\mathcal{D}}=G^{\mathcal{C}}=s$. By the characterization theorem, Theorem \ref{unaryS}, both algebras $\mathcal{C}$ and $\mathcal{D}$ are Ramsey, but the Cartesian product $\mathcal{C}\times\mathcal{D}$ is not Ramsey.

\begin{theorem}
Cartesian products, finite or infinite, of Ramsey algebras need not be Ramsey.
\end{theorem}

We can now take the quotient of the Cartesian product $\mathcal{B}$ above by a nonprincipal ultrafilter on $\omega$ to show that the resulting ultrapower is not Ramsey. This also establishes the fact that an algebra elementarily equivalent to a Ramsey algebra need not be Ramsey.

\begin{theorem}
No ultrapower of the predecessor algebra induced by a nonprincipal ultrafilter on $\omega$ is Ramsey.
\end{theorem}
\begin{proof}
We begin by noting a few facts:
\begin{enumerate}
    \item Being a fixed point is a first-order property when the family of functions is finite (a singleton in our case here). It being unique is also a first-order property.
    \item Every point can be sent to a fixed point ($0$ in particular) by finitely many applications of $p$ is \emph{not} a first-order property. The statement ``finitely-many" does not specify the exact number and it has infinitely many possibilities. We will need countably many disjunctions to express this fact. (This also allows us to give a compactness proof of the theorem, which we will give in the next corollary.)
    \item The point $c$ can be sent to a fixed point by exactly $n$ applications of $p$ is a first-order property, so is it the unique point that can be sent to the unique fixed point by exactly $n$ applications of $p$ a first-order property. This statement differs from the one above by the fact that the finite number $n$ is specified and so a finite number of conjunctions is sufficient to express it.
\end{enumerate}

We will now make use of the first-order properties above. Let $j$ be the canonical embedding of $\mathcal{A}$ into its ultrapower $\ult(\mathcal{A}, \mathcal{U})=\left(\prod_{i\in\omega}\omega\right)/\mathcal{U}$ along a nonprincipal ultrafilter $\mathcal{U}$ on $\omega$. We will call the associated predecessor function in the ultrapower $P$. Firstly, since $0$ is the unique fixed point of the predecessor function $p$, we see that $[j(0)]$ is the unique class that is fixed by $P$. Furthermore, $[j(n)]$ is the unique class that can be sent to $[j(0)]$ by exactly $n$ applications of $P$.

Given $\vec{b}\in\prod_{i\in\omega}\omega$, if $\vec{b}$ can be sent to the unique fixed point $[j(0)]$, then $[\vec{b}]$ must be $\mathcal{U}$-equivalent to $[j(n)]$ for some $n\in\omega$ by the uniqueness we just saw above. Thus, $\vec{b}$ can be an unbounded sequence, but it must be equal to $n$ on a $\mathcal{U}$-measure one set, i.e. $\{i\in\omega:\vec{b}(i)=n\}\in\mathcal{U}$. Now, the sequence $\omega=\langle 0, 1, 2, \ldots\rangle$, which is a member of $\prod_{i\in\omega}\omega$, cannot be $\mathcal{U}$-equivalent to any $[j(n)]$, for that would require the set
\begin{equation*}
\left\{i\in\omega:\omega(i)=n\right\}=\{n\}\in\mathcal{U},
\end{equation*}
which cannot happen since $\mathcal{U}$ is nonprincipal. This shows that $\omega$ cannot be sent to the unique fixed point $j[0]$ by an iterate of $P$. Therefore, the ultrapower $\ult(\mathcal{A}, \mathcal{U})$ is not Ramsey by the characterization of Theorem \ref{unaryS}. Since $\mathcal{U}$ is an arbitrary nonprincipal ultrafilter, the statement is true for any ultrapower along a nonprincipal ultrafilter $\mathcal{U}$ on $\omega$.
\end{proof}

\begin{corollary} \label{firstcompactness}
The notion of a Ramsey algebra is not necessarily preserved under elementary equivalence.
\end{corollary}
\begin{proof}
The corollary follows immediately from the elementary equivalence of an ultrapower. However, we provide an alternative proof using the compactness theorem here. This will also illustrate the point made in Fact 2 of the preceding proof. We will construct an elementarily equivalent algebra $\mathcal{B}$ for which some points of its universe cannot be sent to some fixed points of its associated operation $\phi$.

Let $\mathcal{L}=\{F\}$ be the language of the algebra $\mathcal{A}$. Augment $\mathcal{L}$ with the constant symbol $\zeta, c$ and call the language $\mathcal{L}^*$. Consider the $\mathcal{L}^*$-theory $T=\Th(\mathcal{A})\cup\{F(\zeta)=\zeta\}\cup\{F^i(c)\neq\zeta:i\in\omega\}$. Note that we immediately have $T\models c\neq\zeta$.

Let $\Delta$ be a finite subset of $T$. If $\Delta\cap\{F^i(c)\neq\zeta:i\in\omega\}\neq\varnothing$, let $n\in\omega$ be greatest such that $F^n(c)\neq\zeta\in\Delta$. Clearly, interpreting $\zeta$ by $0$ and interpreting $c$ by $n+1$, the expansion $\mathcal{A}_n$ of $\mathcal{A}$ to the expanded language $\mathcal{L}^*$ is a model of $\Delta$. Hence, by compactness, $T$ is satisfiable. Fix one such model and call it $\mathcal{B}^*$.

Now, $\mathcal{A}\models\exists!x(F(x)=x)$, so $\exists!x(F(x)=x)\in T$, whereby $\mathcal{B}^*\models\exists!x(F(x)=x)$. In particular, this implies that $T\models\forall x(F(x)=x\rightarrow x=\zeta)$. Now, according to $\Delta$, the point $c^{\mathcal{B}^*}\in B$ is such that $\phi^i(c^{\mathcal{B}^*})\neq\zeta^{\mathcal{B}^*}$ for each $i\in\omega$. This says that $c^{\mathcal{B}^*}$ cannot be sent to the only fixed point $\zeta^{\mathcal{B}^*}$ by finitely many applications of $\phi$. As such, the reduct $\mathcal{B}=(B, \phi)$ of $\mathcal{B}^*$ to $\mathcal{L}$ is clearly an algebra in which $\mathcal{B}\equiv \mathcal{A}$ and $\phi^i(c^{\mathcal{B}^*})\neq\zeta^{\mathcal{B}^*}$ for each $i\in\omega$, so $\mathcal{B}$ is not a Ramsey algebra by Theorem \ref{unaryS}.
\end{proof}

Before we end this section, we state two propositions about two special cases of Cartesian products. These special cases deal primarily with finite Ramsey algebras and they can be found in Section 6 of the thesis \cite{teohthesis}. The following theorem, which is part of Theorem 3.9 of \cite{teh2016ramsey}, will be applied. It should be reminded that Cartesian products are only defined for algebras modeling the same language $\mathcal{L}$.

\begin{theorem}\label{finitealg}
Suppose that $(A, \mathcal{F})$ is a finite algebra. Then it is Ramsey if and only if for each $\vec{b}\in{^\omega}A$, there exists $\vec{a}\in{^\omega}A$ such that $\vec{a}\leq_\mathcal{F}\vec{b}$ and $\fr_\mathcal{F}(\vec{a})$ is a singleton.
\end{theorem}

\begin{proposition}
The Cartesian product of finitely many finite Ramsey algebras is a Ramsey algebra.
\end{proposition}
\begin{proof}
It suffices to show that the statement holds for the Cartesian product of two Ramsey algebras.

Assume that $\mathcal{A}_1=(A_1, \mathcal{F})$ and $\mathcal{A}_2=(A_2, \mathcal{G})$ are finite Ramsey algebras. Denote by $\leq$ the reduction relation relation with respect to $\mathcal{F}\times\mathcal{G}$.

Given $X\subseteq A_1\times A_2$ and $\vec{b}\in{^\omega}(A_1\times A_2)$, we obtain a homogeneous reduction in two steps. First, we operate on the first coordinates by applying Theorem \ref{finitealg} above and then we apply the same theorem to the second coordinates of the resulting sequence. To be precise, Theorem \ref{finitealg} furnishes a sequence $f_0, f_1, f_2, \ldots$ of orderly terms of $\mathcal{F}$ and a desired sequence $\vec{\beta}_0, \vec{\beta}_1, \vec{\beta}_2, \ldots$ of finite sequences of the first coordinates of $\vec{b}$ such that $\vec{\beta}'=\langle f_0(\vec{\beta}_0), f_1(\vec{\beta}_1), f_2(\vec{\beta}_2), \ldots\rangle$ has the property $\fr_\mathcal{F}(\vec{\beta}')=\{\phi\}$ for some $\phi\in A_1$. Applying $(f_0, \id_{A_2})$, $(f_1, \id_{A_2})$, $(f_2, \id_{A_2}), \ldots$ to the terms of $\vec{b}$ corresponding to $\vec{\beta}_0, \vec{\beta}_1, \vec{\beta}_2, \ldots$, we obtain a $\vec{b}'\leq\vec{b}$ for which $\fr_{\mathcal{F}\times\mathcal{G}}(\vec{b}')$ consist of ordered pairs whose first coordinates are all equal to $\phi$.

Furnished by Theorem \ref{finitealg}, we may again obtain a sequence $(\id_{A_1}, g_0)$, $(\id_{A_1}, g_1)$, $(\id_{A_1}, g_2), \ldots$ for $\vec{b}'$ to obtain $\vec{a}\leq\vec{b}'\leq\vec{b}$ such that $\fr_{\mathcal{F}\times\mathcal{G}}(\vec{a})$ is now a singleton $\{(\phi, \gamma)\}$ for some $\gamma\in A_2$. Consequently, $\mathcal{A}_1\times\mathcal{A}_2$ is also a (finite) Ramsey algebra by the same theorem again.
\end{proof}

Using Theorem \ref{finitealg} on the finite part and then focusing on the infinite part, we obtain the following corollary.

\begin{corollary}
The Cartesian product of an infinite Ramsey algebra with a finite number of finite Ramsey algebras is a Ramsey algebra.
\end{corollary}

\section{Subalgebras \& Extensions}
A subalgebra of a given algebra is a subset of the universe of the algebra closed under the associated operations. To be precise, let $\mathcal{A}=(A, \mathcal{F})$ be an algebra, let $A'\subseteq A$, and let $\mathcal{F}'=\{f':f'=f\upharpoonright A', f\in\mathcal{F}\}$. Then $\mathcal{A}'=(A', \mathcal{F}')$ is said to be a \emph{subalgebra} of $\mathcal{A}$ if $A'$ is closed under all operations $f'\in\mathcal{F}'$. In what follows, the subalgebra relation will be invariably denoted by the subset relation $\subseteq$, e.g. $\mathcal{A}'\subseteq\mathcal{A}$.
If $(A', \mathcal{F}')$ is a subalgebra and $(A\setminus A', \{f\upharpoonright(A\setminus A'):f\in\mathcal{F}\})$ happens to be a subalgebra as well, then we express this subalgebra as $\mathcal{A}\setminus\mathcal{A}'$. A subalgebra of an algebra is said to be \emph{proper} if the universe of the former is a proper subset of the latter; a subalgebra of an algebra $\mathcal{A}$ is \emph{nontrivial} if it is not the empty algebra nor is it $\mathcal{A}$ itself.

It is intuitively clear that, if an algebra is Ramsey, any subalgebra of it is also Ramsey. Such is indeed the case (cf. \cite{teoh2017matrices}):

\begin{theorem}\label{Ramseysubalgebra}
Every subalgebra $\mathcal{A}'$ of a Ramsey algebra $\mathcal{A}$ is a Ramsey algebra.
\end{theorem}

Whether the converse or Theorem \ref{Ramseysubalgebra} is true remains an open question. If $\mathcal{A}=(A, \mathcal{F})$ is an algebra and $\mathcal{A}'$ is a proper subalgebra of $\mathcal{A}$, it is not necessarily true that $\mathcal{A}\setminus\mathcal{A}'$ is a subalgebra. Some functions in $\mathcal{F}$ may send tuples of elements of $A\setminus|\mathcal{A}'|$ into $|\mathcal{A}'|$.

Let us look at two sufficient conditions to ensure that an algebra is Ramsey whenever its nontrivial subalgebras are Ramsey. These sufficient conditions can be stated in some topological terms.

\begin{definition}
Let $\mathcal{A}=(A, \mathcal{F})$ be an algebra. We define a topology $\subalg(\mathcal{A})$ on $A$ by specifying the basic open sets to be precisely those subsets of $A$ that are the universes of the subalgebras of $\mathcal{A}$. Denote by $\mathfrak{B}(\mathcal{A})$ the set of basic open sets.
\end{definition}

Indeed, $\mathfrak{B}(\mathcal{A})$ forms a basis and it is in fact a Moore collection of subsets of $|\mathcal{A}|$; its members are closed under arbitrary intersections. As such, $\subalg(\mathcal{A})$ is Alexandroff.

\begin{proposition}
Let $\mathcal{A}=(A, \mathcal{F})$ be an algebra such that all of its proper subalgebras are Ramsey. Then $\mathcal{A}$ is also Ramsey if there exists a clopen $A'\in\mathfrak{B}(\mathcal{A})\setminus\{A, \varnothing\}$.
\end{proposition}
\begin{proof}
Let $\vec{b}\in{^\omega}A$, $X\subseteq A$, and let $A'\in\mathfrak{B}(\mathcal{A})\setminus\{A, \varnothing\}$ be clopen. Therefore, $A'$ and $A\setminus A'$ are Ramsey algebras.

By Pigeonhole, pick a subsequence $\vec{b}'$ of $\vec{b}$ all of whose terms are either members of $A'$ or otherwise. That is, $\vec{b}'(i)\in A'$ for each $i\in\omega$ or $\vec{b}'(i)\in A\setminus A'$ for each $i\in\omega$. In either case, a reduction $\vec{a}$ of $\vec{b}'$ that is homogeneous to $X$ can be found. Hence, since $\vec{b}$ and $X$ are arbitrary, it follows that $\mathcal{A}$ is a Ramsey algebra.
\end{proof}

\begin{proposition}\label{compact}
Suppose that $\mathcal{A}=(A, \mathcal{F})$ is an algebra such that all of its proper subalgebras are Ramsey and $\bigcup_{A'\in\mathfrak{B}(\mathcal{A})\setminus\{A\}}A'=A$. If $\subalg(\mathcal{A})$ is compact, then $\mathcal{A}$ is a Ramsey algebra.
\end{proposition}
\begin{proof}
Let $\vec{b}\in{^\omega}A$ and $X\subseteq A$. Let $A_1, \ldots, A_n$ be a finite subcover of $\bigcup_{A'\in\mathfrak{B}(\mathcal{A})\setminus\{A\}}A'=A$. By Pigeonhole, there exists a subsequence $\vec{b}'$ of $\vec{b}$ all of whose terms belong in $A_k$ for some $k\in\{1, \ldots, n\}$. Then by the fact that $A_k$ along with the restricted functions form a Ramsey algebra, the sequence $\vec{b}'$ has a reduction $\vec{a}$ homogeneous for $X$. Therefore, $\vec{b}$ has a reduction homogeneous for $X$. Since $\vec{b}$ and $X$ are arbitrary, it follows that $\mathcal{A}$ is a Ramsey algebra.
\end{proof}

\begin{example}
Let us take a look at the topology $\subalg(\mathcal{A})$ for the predecessor algebra $\mathcal{A}$. The topology is very coarse because the only subalgebras are the empty algebra, the algebra $\mathcal{A}$ itself, and the algebra $(\{0\}, p)$. The topology does not have nontrivial clopen sets, but it is trivially compact. In contrast, let us also take a look at the topology $\subalg(\mathcal{B})$ for $\mathcal{B}$ constructed in the proof of Corollary \ref{firstcompactness}. The empty algebra, the subalgebras $\mathcal{A}$, $\mathcal{B}$ itself, and $\mathcal{B}\setminus\mathcal{A}$ are all members of $\mathfrak{B}(\mathcal{B})$. The open sets $\mathcal{A}, \mathcal{B}\setminus\mathcal{A}$ are a pair of nontrivial clopen basis sets that is dual to each other, but note that $\mathcal{A}$ is Ramsey but $\mathcal{B}\setminus\mathcal{A}$ is not. There is nothing much we can say about compactness unless we delve into the specifics of $\mathcal{B}$, considering the types it can realize, but we will not do so here.
\end{example}

We end our discussion of the topology with a few observations about $\subalg(\mathcal{A})$ in the case when $\mathcal{A}$ is a unary system. For any such algebra, the set
\begin{equation*}
    S=\{c\in A:f(c)=c\;\text{for each}\;f\in\mathcal{F}\}
\end{equation*}
of fixed points for $\mathcal{F}$ is a subalgebra of $\mathcal{A}$, i.e. $S\in\mathfrak{B}(\mathcal{A})$. In fact, the subspace topology on $S$ induced by the topology $\subalg(\mathcal{A})$ is discrete. When $\mathcal{A}$ is Ramsey, we can say even more about $S$; this is in fact a reformulation of Theorem \ref{unaryS}:

\begin{theorem}[Characterization of Ramsey Unary Algebras, topological formulation]
Let $\mathcal{A}=(A, \mathcal{F})$ be a unary system equipped with the topology $\subalg(\mathcal{A})$. Then $\mathcal{A}$ is a Ramsey algebra if and only if the set $S$ of fixed points is dense in $A$.
\end{theorem}
\begin{proof}
Our reference theorem is again Theorem \ref{unaryS}. Suppose $\mathcal{A}$ is Ramsey and let $a\in A$. Then any open set containing $a$ will have a nonempty intersection with $S$. This is because every open set will contain the smallest open set containing $a$, which is the subalgebra generated by $a$. But since $a$ can be sent into $S$ by finitely many applications of the members of $\mathcal{F}$, the subalgebra generated by $a$ has a nonempty intersection with $S$.

Conversely, suppose that $S$ is dense in $\subalg(\mathcal{A})$. Let $a\in A$ and let $U$ be the subalgebra generated by $a$. Then, by $U\in\subalg(\mathcal{A})$ and density, we have $U\cap S\neq\varnothing$. Thus, if $c\in U\cap S$, then there is a way to send $a$ to $c\in S$ using finitely many members of $\mathcal{F}$.
\end{proof}

We continue to explore the converse question, now with a specific example that is somewhere ``in between" a theorem that we are after. There is some triviality involved and it is the reason we considered the smaller collection $\bigcup_{A'\in\mathfrak{B}(\mathcal{A})\setminus\{A\}}A'=A$ in Proposition \ref{compact} above. Consider the algebra $\mathcal{Z}=(\mathbb{Z}, f, g)$, where $f(x)=x+1$ and $g(x)=x-1$.\footnote{This example of an algebra possessing only trivial subalgebras, i.e. the empty algebra and the algebra itself, is given by Qiaochu Yuan.} Then it is easy to check that the only subalgebras of $\mathcal{Z}$ are trivial and that $\mathcal{Z}$ is not a Ramsey algebra by Theorem \ref{unaryS}. Thus, the relevant question is whether an algebra must be Ramsey if nontrivial subalgebras of it exist and if all of them are Ramsey. This question still seems to be elusive. The algebra $\mathcal{D}$ that we will now study is not Ramsey, has countably many isomorphic copies of itself embedded in within, and these isomorphic subalgebras are, therefore, not Ramsey by the isomorphism theorem above. Nevertheless, $\mathcal{D}$ is an algebra for which all \emph{other} subalgebras are Ramsey. We begin the construction of $\mathcal{D}$ by looking at one-point extensions of Ramsey algebras.

Suppose that $\mathcal{A}=(A, \mathcal{F})$ is a Ramsey algebra. Consider the new universe $A'=A\cup\{\alpha\}$ augmented by a new symbol $\alpha$; let $\mathcal{F}'$ be a family of operations on $A$, each member $f'$ being a fixed arbitrary extension of some member $f\in\mathcal{F}$ to $A'$. We claim that the new algebra $\mathcal{A}'=(A', \mathcal{F}')$ is also a Ramsey algebra.

To see this, let $\vec{b}\in{^\omega}A'$ and $X\subseteq A'$. If $\alpha$ appears only finitely many times in the terms of $\vec{b}$, we may drop those terms to obtain $\vec{b}'$ and, clearly, $\vec{b}'\leq_{\mathcal{F}'}\vec{b}$. Since $\vec{b}'$ consists of terms belonging in $A$ and $\mathcal{A}$ is a Ramsey algebra, $\vec{b}'$ clearly has a reduction $\vec{a}\leq\vec{b}'$ homogeneous for $X$ and so $\vec{b}$ has a reduction $\vec{a}\leq_{\mathcal{F}'}\vec{b}$ homogeneous for $X$.

On the other hand, suppose that $\alpha$ occurs infinitely many times in $\vec{b}$. We drop the other terms and obtain the subsequence $\vec{b}'$, which is a constant sequence all of whose terms are $\alpha$. As usual, $\vec{b}'\leq_{\mathcal{F}'}\vec{b}$, so it suffices to show the existence of a reduction of $\vec{b}'$ homogeneous for $X$. There are two possibilities to consider:
\begin{enumerate}
    \item All $F'\in\ot(\mathcal{F}')$ yield the value $\alpha$ when given input $(\alpha, \ldots, \alpha)$.
    \item There exists an $F'\in\ot(\mathcal{F}')$ such that $F'(\alpha, \ldots, \alpha)\neq\alpha$, thus $F'(\alpha, \ldots, \alpha)\in A$.
\end{enumerate}
In the former case, $\vec{b}'$ itself is homogeneous for $X$; in fact, $\fr_{\mathcal{F}'}(\vec{b}')=\{\alpha\}$, and so $\vec{b}'$ is clearly homogeneous for $X$. In the latter case, apply $F'$ on consecutive blocks of $\langle\alpha, \ldots, \alpha\rangle$ (of length the arity of $F'$) in $\vec{b}'$ so that we obtain the reduction $\vec{a}'\leq_{\mathcal{F}'}\vec{b}'$ all of whose terms lie in $A$. Since $\vec{a}'\in{^\omega}A$ and $\mathcal{A}$ is a Ramsey algebra, we have a reduction $\vec{a}\leq_{\mathcal{F}'}\vec{a}'$ that is homogeneous for $X$.

Dovetailing the observations above, we obtain the following lemma:
\begin{lemma}
Any one-point extension of a Ramsey algebra also results in a Ramsey algebra.
\end{lemma}

The notion of a direct limit is one that is familiar from model theory (cf. \cite{changkeisler1990model}) and we will not give it here.

\begin{theorem}\label{dirlimthm}
The direct limit of a sequence of Ramsey algebras need not be a Ramsey algebra.

\end{theorem}
\begin{proof}
We recursively construct a $\subseteq$-chain of algebras. Begin with $\mathcal{A}_0=(A_0, f)=(\{0\}, f_0)$, where $f_0$ is a binary function; it is clearly a Ramsey algebra. By recursion, we let $\mathcal{A}_{n+1}=(A_{n+1}, f_{n+1})=(A_n\cup\{n+1\}, f_n)$, where $f_{n+1}$ denotes the extension of $f_n$ from $A_n$ to $A_{n+1}$ given by
\begin{equation*}
   f_{n+1}(a, b) = \left\{
  \begin{array}{ll}
    n+1, & \hbox{$a=b=n+1$;} \\
    n, & \hbox{$a\neq b$ while either one equals $n+1$.}
  \end{array}
\right.
\end{equation*}
We denote the direct limit of these algebras by $\mathcal{D}$, the universe of the algebra being $\omega$ and the operation we denote by $f$. Note that each natural number is an idempotent element of $\mathcal{D}$, i.e. $f(n, n)=n$ for each $n\in\omega$.

Note that $f$ so defined on $\omega$ is not associative; for instance, we have $f(f(0, 0), 3)=2\neq 1=f(0, f(0, 3))$. Hence $\mathcal{D}$ is not a semigroup, whence Hindman's theorem clearly does not apply. In fact, we now show that, while each $\mathcal{A}_n$ is a Ramsey algebra because each is a one-point extension of the previous Ramsey algebra, the direct limit $\mathcal{D}$ is not Ramsey. We will show that $\vec{b}=\langle 0, 1, 2, \ldots\rangle$ does not have a reduction homogeneous for the set $X$ of even numbers. We first need the following claim.

\begin{claim}
If $F\in\ot(\{f\})$, $\sigma$ is a finite subsequence of $\vec{b}=\langle 0, 1, \ldots\rangle$, the first and last terms of $\sigma$ is $N$ and $M$, respectively, then $N\leq F(\sigma)<M$.
\end{claim}
\textit{Proof of Claim.} By induction on the complexity of $F$. For the atomic case, we have $F=f$. Let $\sigma=(N, M)$ be a finite subsequence of $\vec{b}$; we have $0\leq N<M$. Then we have $0\leq F(\sigma)=f(N, M)=M-1<M$. Now, suppose that the claim is true for the orderly terms $F_1, F_2$ the finite subsequences $\sigma_1, \sigma_2$ with first and last term pairs equaling $(N_1, M_1)$ and $(N_2, M_2)$, respectively, and with $M_1<N_2$. We want to show that the orderly term next up in complexity, namely $F(\bar{x}_1\ast\bar{x}_2)=f(F_1(\bar{x}_1), F_2(\bar{x}_2))$, also satisfies the statement of the claim. Indeed, by induction hypothesis, we have $N_1\leq F_1(\sigma_1)<M_1$ and $N_2\leq F_2(\sigma_2)<M_2$. Now, since $M_1<N_2$, we have $F_1(\sigma_1)<F_2(\sigma_2)$, whence $N_1\leq f(F_1(\sigma_1), F_2(\sigma_2))=F_2(\sigma_2)-1<M_2$ as desired. \qed (Claim.)

From the claim, it is easy to see that every reduction $\vec{a}\leq\vec{b}$ is an infinite subsequence of $\vec{b}$. In particular, this implies that, for each of these reductions $\vec{a}$, we have $\vec{a}(1), f(\vec{a}(0), \vec{a}(1))\in\fr(\{f\})$, but these two numbers have different parity, whence $\vec{a}$ cannot be homogeneous for $X$. This concludes the proof of the theorem.
\end{proof}

In the spirit of the converse question above, let us inspect the subalgebras of $\mathcal{D}$:

\begin{proposition}
Every nontrivial subalgebra of the direct limit $\mathcal{D}$ is some $\mathcal{A}_n$, some $\mathcal{D}\setminus\mathcal{A}_n$, or some $\mathcal{A}_n\setminus\mathcal{A}_m$.
\end{proposition}
\begin{proof} We first note that structures of the forms $\mathcal{A}_n$, $\mathcal{D}\setminus\mathcal{A}_n$, and $\mathcal{A}_n\setminus\mathcal{A}_m$ are subalgebras of $\mathcal{D}$. We now show that these are the only proper subalgebras.

Let $W\subsetneq\omega$ be nonempty and let $w_0$ be least in $W$; we claim that $W$ has one of the stipulated forms. If $W$ is a singleton, then $W$ generates the algebra $(\{w_0\}, f)$, but $(\{w_0\}, f)$ is precisely $\mathcal{A}_{w_0}\setminus\mathcal{A}_{w_0-1}$ in the case $w_0>0$ or $\mathcal{A}_0$ in the case $w_0=0$.

On the other hand, suppose that $W$ is not a singleton. Let $w$ be a nonzero member of $W$ distinct from $w_0$. It is easy to see from the definition of $f$ that $\{w_0, w_0+1, \ldots, w\}\subseteq W$, whence it follows by induction that $\mathcal{A}_w\setminus\mathcal{A}_{w_0}$ is a subalgebra of the algebra generated by $W$. (In fact, the induction begins with the observation that $f(w, w_0)=w-1$ and the induction step takes the form $f(w-k-1, w_0)=w-k-2$, terminating when we hit $w_0$.) But $w$ is arbitrary, hence if the algebra generated by $W$ is nontrivial, then it must have the form $\mathcal{D}\setminus\mathcal{A}_n$ or $\mathcal{A}_n\setminus\mathcal{A}_{w_0}$ in the case $w_0>0$ or $\mathcal{A}_n$ in the case $w_0=0$. This concludes the proof that every nontrivial subalgebra of $\mathcal{D}$ assumes one of the stipulated forms.
\end{proof}

In the proof of Theorem \ref{dirlimthm} above, we have seen that $\mathcal{D}$ is not a Ramsey algebra. Now note that each $\mathcal{D}\setminus\mathcal{A}_n$ is isomorphic to $\mathcal{D}$, so these subalgebras are not Ramsey either by Corollary \ref{isothm}. For the other subalgebras, note that each $\mathcal{A}_n$ and each $\mathcal{A}_n\setminus\mathcal{A}_m$ is a finite algebra. The following theorem on finite algebras, appearing as part of Theorem 3.9 in \cite{teh2016ramsey} and which is related to Theorem \ref{finitealg} above, will show that the finite subalgebras $\mathcal{A}_n$ and $\mathcal{A}_n\setminus\mathcal{A}_m$ are, in fact, Ramsey.

\begin{theorem}\label{finitealgebra}
A finite algebra $(A, \mathcal{F})$ is a Ramsey algebra if and only if every nonempty subalgebra contains an idempotent element $a$ of $(A, \mathcal{F})$, i.e. $f(a, a, \ldots, a)=a$ for all $f\in\mathcal{F}$.
\end{theorem}

Each $\mathcal{A}_n$ as well as each $\mathcal{A}_n\setminus\mathcal{A}_m$ is Ramsey because every natural number is idempotent for $\mathcal{D}$. As such, we have seen that the direct limit $\mathcal{D}$ is not a Ramsey algebra, so aren't the subalgebras $\mathcal{D}\setminus\mathcal{A}_n$ that are isomorphic copies of it, but the remaining subalgebras $\mathcal{A}_n$ and $\mathcal{A}_n\setminus\mathcal{A}_m$ are Ramsey algebras. In summary, we have seen that the ring $\mathbb{Z}$, interpreted as a model $(\mathbb{Z}, +, \times)$ of a purely functional language, is not Ramsey because it has no nontrivial subalgebras that are Ramsey, and $\mathcal{D}$ is not Ramsey despite of the fact that all of its subalgebras that are not isomorphic to it are Ramsey. One final example from octonions to show that if an algebra fails to be Ramsey, then it must already have subalgebras that are not Ramsey. It is shown in \cite{teohoctonions} that the real octonions $(\mathbb{O}, \cdot)$ is not a Ramsey algebra under multiplication. Although the real octonions considered as a normed division algebra has only proper subalgebras isomorphic to $\mathbb{R}$, $\mathbb{C}$, and the quaternions $\mathbb{H}$, there are more subalgebras when considered as a binary system under multiplication alone. Take for instance the subalgebra $\mathcal{B}=(B, \cdot\upharpoonright B)$ generated by the set
\begin{equation*}
    B=\left\{b_n:n\in\omega, b_n=\sum_{i=0}^7 2^{2^{8n+1+i}}e_i\right\},
\end{equation*}
where $e_0, e_1, \ldots, e_7$ are the eight unit octonions whose exact indexing is irrelevant for our purpose. Then, $\mathcal{B}$ is a proper subalgebra of $(\mathbb{O}, \cdot)$ because, for each of the elements of $B$, the coefficients of the unit octonions are rational while $\left(\sum_{i=0}^7\sqrt{2}e_i\right)\not\in B$. In \cite{teohoctonions}, it is shown that the sequence $\vec{b}=\langle b_0, b_1, \ldots\rangle$ and some set $X$ together form a witness to the fact that $(\mathbb{O}, \cdot)$ is not a Ramsey algebra. The set $X$ has the property that, given and $\vec{a}\leq_{\{\cdot\}}\vec{b}$ and any natural numbers $i<j<k$, we have $\vec{a}(i)\cdot(\vec{a}(j)\cdot\vec{a}(k))\in X$ and $(\vec{a}(i)\cdot\vec{a}(j))\cdot\vec{a}(k)\in X^C$. Since $\fr_{\{\cdot\}}(\vec{b})\subseteq B$, we now see that $\vec{b}$, along with $X$, acts as a witness to the fact that $\mathcal{B}$ is not Ramsey. As such, the question posed at the beginning of this section remains open, but we yet see another evidence that the answer could be in the affirmative.

\section{Conclusion}
The results above answer a number of structural questions pertaining to Ramsey algebras. Since isomorphic algebras are essentially the same algebra, it came as no surprise that isomorphic algebras have the same Ramsey algebraic property. On the other hand, elementary equivalence is a first-order property and the invariance of Ramsey property breaks down in the face of the stronger third-order property of being a Ramsey algebra. This also suggests another line of investigation, namely when does a Ramsey algebra have a first-order characterization? Hindman's theorem states that semigroups are Ramsey algebra and being a semigroup is a first-order property. On the contrary, the characterization theorem, Theorem \ref{unaryS}, is not a first-order statement as we saw by means of an example, which is contained in the proof of Theorem \ref{firstcompactness}.

We also answered a long-standing question in the subject as to whether Cartesian products of Ramsey algebras must be Ramsey and the answer turns out to be in the negative, something expected in hindsight. It will now be interesting to investigate if one can always reduce an infinite product to yield a reduced product that is Ramsey, an investigation towards a characterization of the filters that would do the job. The result above on the ultrapowers of the predecessor algebra somehow show that this might not come by easy.

Apart from our results on subalgebras, most of the other theorems we considered involved enlarging a given algebra. Cartesian products are enlargements, so are direct limits. Answering the question about elementarily equivalent algebras, we constructed the ultrapower; alternatively, we also constructed, by way of the compactness theorem, an extension of the given Ramsey algebra. As we have seen, the enlargements that accompany these results do not exhibit Ramsey property. In summary, going from a Ramsey algebra to a larger algebra does not necessarily preserve Ramseyness. This runs in stark contrast with the usual Ramseyan theme, where larger domains of a structure that exhibits Ramseyan properties would remain Ramsey. A remedy, if we would, is to shift perspective from the universe of the algebra to the countable sequences of its universe as they pertain to the definition of a Ramsey algebra. To be precise, consider a sequence $\vec{b}$ for which one is interested to find a homomgeneous reduction $\vec{a}$. If $\vec{b}$ admits a homogeneous reduction $\vec{a}$, then any ``enlarged" sequence $\vec{b}'$ that contains $\vec{b}$ as a subsequence will have $\vec{a}$ as a homogeneous reduction. Thus, the proper ``domain" of enlargement lies in sequences rather than the domain of the algebra in question. At any rate, a further understanding of this issue would be illuminating.

Finally, one nagging question remains: For algebras with an abundance of nontrivial subalgebras, if every subalgebra is Ramsey, must the algebra be Ramsey? One is free to interpret the term ``abundance" in this question.

\section{Acknowledgment}
I would like to thank Xianghui Shi for raising the question of whether enlargements of a Ramsey algebra will remain Ramsey. I also extend my utmost gratitude to Wen Chean Teh for constantly suggesting improvements on writing proofs. The painstaking effort taken by Noor Atinah Ahmad, Kai Lin Ong, and Azhana Ahmad in reading the manuscript and providing valuable feedback is an aspect in the preparation of the manuscript that I truly cherish. Special thanks also goes to Qiaochu Yuan for giving the example algebra $\mathcal{Z}$ above. Finally, I dedicate this paper to my late father Chin Seng Teoh, forever in fond memories.

\end{document}